\documentclass[12pt,reqno]{amsart}
\usepackage{fullpage}
\usepackage{color} \newcommand{\red}[1]{{\color{red}#1}} \definecolor{darkgreen}{rgb}{0,0.6,0}
 \renewcommand{\red}[1]{#1}
\usepackage{times,url}

\newtheorem{theorem}{Theorem}
\newtheorem{lemma}[theorem]{Lemma}
\newtheorem{prop}[theorem]{Proposition}
\newtheorem{cor}[theorem]{Corollary}
\theoremstyle{definition}
\newtheorem{definition}[theorem]{Definition}

\renewcommand{\mod}[1]{{\ifmmode\text{\rm\ (mod~$#1$)}\else\discretionary{}{}{\hbox{ }}\rm(mod~$#1$)\fi}}

\newsymbol\nmid 232D

\newcommand{\legendre}[2]{\genfrac{(}{)}{}{}{#1}{#2}}
\newcommand\floor[1]{\lfloor#1\rfloor}

\renewcommand{\L}{{\mathcal L}}

\begin{document}

\title{Factorization tests and algorithms arising from counting modular forms and automorphic representations}
\author{Miao Gu and Greg Martin}
\address{Mathematics Department \\ Duke University, Box 90320 \\Durham, NC, USA \ 27708--0320}
\email{miao.gu@duke.edu}
\address{Department of Mathematics \\ University of British Columbia \\ Room 121, 1984 Mathematics Road \\ Vancouver, BC, Canada \ V6T 1Z2}
\email{gerg@math.ubc.ca}
\subjclass[2010]{11N25, 11Y16, 11N60, 11Y05, 11F70.}
\maketitle

\begin{abstract}
A theorem of Gekeler compares the number of non-isomorphic automorphic representations associated with the space of cusp forms of weight $k$ on~$\Gamma_0(N)$ to a simpler function of $k$ and~$N$, showing that the two are equal whenever $N$ is squarefree. We prove the converse of this theorem (with one small exception), thus providing a characterization of squarefree integers. We also establish a similar characterization of prime numbers in terms of the number of Hecke newforms of weight $k$ on~$\Gamma_0(N)$.

It follows that a hypothetical fast algorithm for computing the number of such automorphic representations for even a single weight $k$ would yield a fast test for whether $N$ is squarefree. We also show how to obtain bounds on the possible square divisors of a number $N$ that has been found to not be squarefree via this test, and we show how to \red{probabilistically} obtain the complete factorization of the squarefull part of $N$ from the number of such automorphic representations for two different weights. If in addition we have the number of such Hecke newforms for even a single weight $k$, then we show how to probabilistically factor $N$ entirely.
All of these computations could be performed quickly in practice, given the number(s) of automorphic representations and modular forms as input.
\end{abstract}

\section{Introduction}

An integer is {\em squarefree} if it is not divisible by the square of any prime. Deciding whether a number is squarefree is trivial if one has its complete factorization; however, we currently lack fast algorithms for factoring large integers, nor do we have any alternate characterization of squarefree numbers that allows for a faster test (unlike the case of polynomials over a field of characteristic $0$, for example, where a polynomial is squarefree if and only if it is coprime to its derivative). The origin of this paper is an interesting connection, related to squarefreeness, between the number of certain automorphic representations and the value of a very simple function.

\begin{definition} 
Let $A(k,N)$ denote the number of non-isomorphic automorphic representations associated with the space of cusp forms of weight $k$ on~$\Gamma_0(N)$. An explicit formula for $A(k,N)$ as a linear combination of multiplicative functions of $N$, as derived by the second author~\cite{M}, is given in Proposition~\ref{AkN formula} below. \red{(An equivalent way to describe the quantity $A(k,N)$ is the number of weight-$k$ Hecke newforms of level dividing~$N$.)}
\end{definition}

\begin{definition}  \label{GkN def}
For any positive integer $N$ and any positive even integer $k$, define
\[
G(k,N)= \frac{k-1}{12}N-\frac{1}{2}+c_2(k)\legendre {-4} N+c_3(k)\legendre{-3} N.
\]
The quantities $\legendre {-4} N$, $\legendre {-3} N$, $c_2(k)$, and $c_3(k)$ are given in Definitions~\ref{-4N -3N def} and~\ref{c2 c3 def} below; for now, we emphasize that they depend only upon the residue classes of $N$ and $k$ modulo~$12$. Consequently, $G(k,N)$ can be computed extremely rapidly, even without knowing the factorization of~$N$.
\end{definition}

Gekeler~\cite{G} proved that the function $G(k,N)$ (for which he gave a different but equivalent expression) is equal to $A(k,N)$ when $N\ge2$ is squarefree. Our first theorem is a converse of this statement (with one small exception and another small case where the expected inequality is reversed).

\begin{theorem}  \label{main thm}
Let $N\ge2$ be an integer and $k$ a positive even integer.
\begin{itemize}
\item When $N$ is squarefree, or when $k=2$ and $N=9$, we have $G(k,N)=A(k,N)$.
\item When $k=2$ and $N=4$, we have $G(k,N)<A(k,N)$.
\item In all other cases, $G(k,N)>A(k,N)$.
\end{itemize}
\end{theorem}

\begin{cor}  \label{main cor}
Let $N\ge10$. Then for any positive even integer $k$, we have that $N$ is squarefree if and only if $A(k,N) = G(k,N)$.
\end{cor}

We pause to dwell upon some hypothetical significance of this corollary. As of the writing of this paper, nobody has found an algorithm that determines whether a positive integer $N$ is squarefree or not that is significantly faster than factoring $N$; in particular, we do not know any polynomial-time algorithm for testing squarefreeness. The standard way to compute $A(k,N)$ is through factoring $N$ (as per the formula in Proposition~\ref{AkN formula}). However, if someone were to develop an alternate way to calculate $A(k,N)$ that was much faster, then Corollary~\ref{main cor} would provide a fast way to test the number $N$ for squarefreeness. Indeed, it is tantalizing to observe that such an alternate calculation of $A(k,N)$ needn't be particularly robust: {\em a polynomial-time algorithm for calculating, given a number $N$, even one single value of $A(k,N)$---perhaps for a special even number $k$ depending upon $N$---would yield a polynomial-time algorithm for testing whether $N$ is squarefree} (as long as $k$ were not astronomically large). {We could even obtain the same outcome with a fast algorithm yielding a sufficiently good upper bound for $A(k,N)$, or one that calculated a positive linear combination of $A(k,N)$ for several values of~$k$.} \red{It is admittedly difficult to speculate about what such an algorithm would entail: any method that actually enumerated Hecke eigenforms, for example, would be slower than factoring $N$ in practice because the number of such eigenforms is essentially linear in~$N$ (and hence exponentially large in the length of~$N$).}

One can extract more information from this idea than simply whether $N$ is squarefree. For example, in Proposition~\ref{trig bounds prop} below, we give upper and lower bounds (depending upon the difference between $G(k,N)$ and $A(k,N)$ for a single value of $k$) for the size of any integer $d$ whose square divides~$N$. If we have access to two distinct values of $A(k,N)$ for the same number $N$, we can do even better, as the next theorem demonstrates. Recall that an integer is {\em squarefull} if every prime dividing it divides it to at least the second power; every number $N$ has a unique factorization of the form $N=EL$ where $E$ is squarefree, $L$ is squarefull, and $\gcd(E,L)=1$.

\begin{theorem}  \label{2 inputs thm}
Let $N$ be a positive integer. Suppose we know two values $A(k_1,N)$ and $A(k_2,N)$ for distinct positive even integers $k_1$ and $k_2$. Then we can quickly obtain the complete factorization of the squarefull part of~$N$. More precisely, we can, in probabilistic polynomial time, calculate distinct primes $p_1,\dots,p_\ell$, integers $e_1,\dots,e_\ell \ge 2$, and a squarefree number $E$ that is relatively prime to $p_1\cdots p_\ell$, satisfying $N=Ep_1^{e_1}\cdots p_\ell^{e_\ell}$.
\end{theorem}

\noindent (The theorem is valid, though uninteresting, when $N$ itself is squarefree, since $\ell=0$ is permitted. We remark that it would suffice to have two values $A(k_1,M)$ and $A(k_2,M)$ for any multiple $M$ of $N$, since one can easily deduce the factorization of the squarefull part of $N$ from the factorization of the squarefull part of~$M$.) We do not explicitly report the complexity of the polynomial-time algorithms in this paper, although many of them are extremely fast in practice.

\red{As far as we are aware, these results represent the first known applications of enumerative results in the theory of automorphic representations to computational complexity questions related to integer factorization. In this vein, it seems worth pointing out a similar application of the dimension of the space of cusp forms on $\Gamma_0(N)$ to primality testing (even though, in contrast to deciding whether a number is squarefree, primality testing is already in quite an acceptable state).}

\begin{definition}  \label{B def}
Let $B(k,N)$ denote the dimension of the space of weight-$k$ newforms on $\Gamma_0(N)$. (The function $B(k,N)$ is often denoted by $g_0^\#(k,N)$; an explicit formula for this function as a linear combination of multiplicative functions of $N$ was given by the second author~\cite[Theorem~1]{M} and is summarized in Proposition~\ref{B prop} below.)
\end{definition}

\begin{definition}  \label{H def}
Define the function
\begin{equation}
H(k,N) = G(k,N) - B(k,1) = G(k,N) - \bigg( \frac{k-7}{12} + c_2(k) + c_3(k) + \delta_2(k) \bigg),
\end{equation}
where $G(k,N)$ is as in Definition~\ref{GkN def}; note that $H(k,N)$ can be computed extremely rapidly, even without knowing the factorization of~$N$.
\end{definition}

Our second theorem demonstrates that a polynomial-time algorithm for calculating $B(k,N)$ would yield a very fast algorithm for testing whether $N$ is prime:

\red{\begin{theorem}  \label{prm test thm}
Let $N\ge2$ be an integer and $k$ a positive even integer.
\begin{itemize}
\item When $N$ is prime, or when $k=4$ and $N=6$, or when $k=2$ and $N=6$, $9$, $10$, $14$, $15$, $21$, $26$, $35$, $39$, $65$, or $91$, we have $H(k,N)=B(k,N)$.
\item When $k=2$ and $N=4$, we have $H(k,N)<B(k,N)$.
\item In all other cases, $H(k,N)>B(k,N)$.
\end{itemize}
\end{theorem}}

\red{\begin{cor}  \label{prm test cor}
Let $N\ge92$. Then for any positive even integer $k$, we have that $N$ is prime if and only if $H(k,N)=B(k,N)$.
\end{cor}}

Of course, a deterministic polynomial-time primality test already exists~\cite{AKS}, and we have very fast probabilistic primality tests (although, depending upon the speed of the hypothetical oracle that calculates $B(k,N)$, the test resulting from Corollary~\ref{prm test cor} could be even faster). However, if we combine the ideas from the proofs of the previous theorems, we can actually produce a fast method for factoring integers:

\begin{theorem}  \label{3 inputs thm}
Let $N$ be a positive integer. Suppose we know two values $A(k_1,N)$ and $A(k_2,N)$ for distinct positive even integers $k_1$ and $k_2$, and a value $B(k,N)$ for some positive even integer $k$. Then we can calculate the complete factorization of~$N$ in probabilistic polynomial time.
\end{theorem}

\noindent As remarked after Theorem~\ref{2 inputs thm}, it would suffice to know values $A(k_1,M)$, $A(k_2,M)$, and $B(k,M)$ for any multiple $M$ of~$N$.

We establish Theorem~\ref{main thm} in the next section. In Section~\ref{bounds section} we establish upper and lower bounds for square divisors of $N$ based on the difference between $G(k,N)$ and $A(k,N)$. Thereafter, we prove Theorem~\ref{2 inputs thm} in Section~\ref{two values section}, Theorem~\ref{prm test thm} in Section~\ref{prm thm section}, and Theorem~\ref{3 inputs thm} in Section~\ref{fact N section}.

\section{Testing for squarefreeness}  \label{main thm section}

In this section we establish Theorem~\ref{main thm}.
We begin by giving an explicit formula (Proposition~\ref{AkN formula}) for $A(k,N)$; to do so, we must start with several definitions of functions appearing in that formula, as well as in Definition~\ref{GkN def} for $G(k,N)$. After establishing sufficient notation and recording a useful lower bound for $G(k,N)-A(k,N)$, we establish Theorem~\ref{main thm} via Lemmas~\ref{5 and up lemma}--\ref{2 lemma}.

\begin{definition}  \label{-4N -3N def}
$\legendre {-4} N$ and $\legendre {-3} N$ are special values of the Kronecker symbol:
\[
\legendre {-4} N = 
  \begin{cases}
    1, &  \text{if } N \equiv 1 \mod 4,\\
    -1, &  \text{if } N \equiv 3 \mod 4,\\
    0, & \text{if } 2 \mid N;
  \end{cases}
\qquad
\legendre {-3} N = 
  \begin{cases}
    1, &  \text{if } N \equiv 1 \mod 3,\\
    -1, &  \text{if } N \equiv 2 \mod 3,\\
    0, &  \text{if } 3\mid N.
  \end{cases}
\]
\end{definition} 

\begin{definition} \label{c2 c3 def}
The functions $c_2$ and $c_3$ are defined as follows:
\begin{align*}
c_2(k) &= \frac{1}{4}+\bigg\lfloor{\frac{k}{4}}\bigg\rfloor-\frac{k}{4} = 
  \begin{cases}
    {1}/{4}, &  \text{if } k \equiv 0 \mod 4,\\
    -{1}/{4}, &  \text{if } k \equiv 2 \mod 4 ;
  \end{cases} \\
c_3(k) &= \frac{1}{3}+\bigg\lfloor{\frac{k}{3}}\bigg\rfloor-\frac{k}{3} =
  \begin{cases}
    {1}/{3}, &  \text{if } k \equiv 0 \mod 3,\\
    0,           & \text{if }  k \equiv 1 \mod 3,\\
    -{1}/{3}, &  \text{if } k \equiv 2 \mod 3 .
  \end{cases}
\end{align*}
(We do not list the values of $c_2(k)$ when $k$ is odd since we consider only even integers $k$ in this paper.)
We also define
\[
\delta_1(m) =   \begin{cases}
    1, &  \text{if } m=1,\\
    0, &  \text{if } m\ne1
  \end{cases}
\quad\text{and}\quad
\delta_2(m) =   \begin{cases}
    1, &  \text{if } m=2,\\
    0, &  \text{if } m\ne2.
  \end{cases}
\]
\end{definition} 

\begin{definition} \label{mult fns def}
The multiplicative functions $s_0^{*}$ and $\nu_{\infty}^{*}$ are defined as follows:
\[s_0^{*}(N) =
\prod_{\substack{p^{e}\parallel N \\ e\ge2}} \bigg(1-\frac{1}{p^2}\bigg) ;
\qquad
\nu_{\infty}^{*}(N) =
\prod_{\substack{p^{e}\parallel N \\ e\ge2}} (p-1)p^{\floor{{e}/2-1}}.
\]
In particular, $s_0^{*}(N)=\nu_{\infty}^{*}(N)=1$ when $N$ is squarefree. Note that if $M\mid N$, then $s_0^*(M) \ge s_0^*(N)$ and $\nu_\infty^*(M) \le \nu_\infty^*(N)$; in particular, $s_0^*(N) \le 1 \le \nu_\infty^*(N)$ for all positive integers~$N$. We also remark that if $D$ is the largest integer such that $D^2\mid N$, then $\nu_{\infty}^{*}(N) = \phi(D)$, where $\phi$ is the Euler phi-function.
\end{definition} 

\begin{definition}  \label{nu2 nu3 def}
The multiplicative functions $\nu_2^{*}$ and $\nu_3^{*}$ are defined in terms of the Kronecker symbol (see Definition~\ref{-4N -3N def}) as follows:
\begin{align*}
\nu_2^*(N) &=
\begin{cases}
     \legendre {-4} {N},  & \text{if $N$ is squarefree}, \\
     -\legendre {-4} {N/4},  & \text{if $4\mid N$ and ${N}/{4}$ is squarefree}, \\
    0, & \text{otherwise};
\end{cases} \\
\nu_3^*(N) &=
\begin{cases}
     \legendre {-3} {N},  & \text{if $N$ is squarefree}, \\
     -\legendre {-3} {N/9},  & \text{if $9\mid N$ and ${N}/{9}$ is squarefree}, \\
    0, & \text{otherwise}.
\end{cases}
\end{align*}
\end{definition} 

The following proposition was derived by the second author~\cite[Theorem 4]{M}. (In that paper, the function $A(k,N)$ was denoted by $g_{\red 0}^*(k,N)$, but that notation would be more confusing in the present context. It can be quickly verified that the formulas given in Definitions~\ref{mult fns def} and~\ref{nu2 nu3 def} are equivalent to those given in~\cite[Definition 4$'$]{M}.)

\begin{prop} \label{AkN formula}
For any integer $N\ge2$ and any even integer $k\ge2$,
\[
A(k,N)=\frac{k-1}{12}Ns_0^{*}(N)-\frac{1}{2}\nu_{\infty}^{*}(N)+c_2(k)\nu_2^{*}(N)+c_3(k)\nu_3^{*}(N).
\]
\end{prop}

We now characterize the values of $k$ and $N$ for which the actual number $A(k,N)$ of \red{non-isomorphic} automorphic \red{representations} equals the simpler function $G(k,N)$ from Gekeler's theorem.

\begin{definition}  \label{Delta def}
Define $\Delta(k,N)=G(k,N)-A(k,N)$. From Definition~\ref{GkN def} and Proposition~\ref{AkN formula}, we see that for any integer $N\ge2$ and any even integer $k\ge2$,
\begin{multline} \label{Deltakn formula}
\Delta(k,N) = \frac{k-1}{12} N \big( 1 - s_0^*(N) \big) + \frac12\big( \nu_\infty^*(N) - 1 \big) \\
+ c_2(k) \bigg( \legendre{-4}N - \nu_2^*(N) \bigg) + c_3(k) \bigg( \legendre{-3}N - \nu_3^*(N) \bigg).
\end{multline}
\end{definition}

Our intuition should be that square divisors of $N$ cause the first two terms on the right-hand side of equation~\eqref{Deltakn formula} to be significantly positive. We proceed to make this strategy precise.

\begin{lemma}  \label{Deltakn ineq lemma}
For any integer $N\ge2$ and any even integer $k\ge2$,
\begin{equation}  \label{1312 ineq}
\Delta(k,N) \ge \frac{k-1}{12} N \big( 1 - s_0^*(N) \big) + \frac12\nu_\infty^*(N) - \frac{13}{12}.
\end{equation}
\end{lemma}

\begin{proof}
We easily verify that
\begin{equation*}
\begin{split}
\legendre{-4}N - \nu_2^*(N) &= \begin{cases}
     0,  & \text{if $N$ is squarefree}, \\
     \legendre {-4} {N/4},  & \text{if $4\mid N$ and ${N}/{4}$ is squarefree}, \\
    \legendre{-4}N, & \text{otherwise};
\end{cases} \\
\legendre{-3}N - \nu_3^*(N) &= \begin{cases}
     0,  & \text{if $N$ is squarefree}, \\
     \legendre {-3} {N/9},  & \text{if $9\mid N$ and ${N}/{9}$ is squarefree}, \\
    \legendre{-3}N, & \text{otherwise}.
\end{cases}
\end{split}
\end{equation*}
In particular,
\[
\bigg| c_2(k) \bigg( \legendre{-4}N - \nu_2^*(N) \bigg) \bigg| \le \frac14 \quad\text{and}\quad \bigg| c_3(k) \bigg( \legendre{-3}N - \nu_3^*(N) \bigg) \bigg| \le \frac13.
\]
The inequality~\eqref{1312 ineq} now follows immediately from the formula~\eqref{Deltakn formula}.
\end{proof}

In the current notation, Theorem~\ref{main thm} characterizes the sign of $\Delta(k,N)$ in terms of $k$ and~$N$. Gekeler's theorem already tells us that $\Delta(k,N)=0$ when $N\ge2$ is squarefree; this fact can be quickly verified by noting that all four summands on the right-hand side of equation~\eqref{Deltakn formula} vanish when $N$ is squarefree, thus reproving Gekeler's theorem as a corollary of Proposition~\ref{AkN formula}.

At this point, then, to establish Theorem~\ref{main thm}, it remains only to prove that if $N$ is not squarefree then $\Delta(k,N)>0$, except for the two exceptions $\Delta(2,9)=0$ and $\Delta(2,4)=-\frac12$. We accomplish this via the next three lemmas, distinguished by the size of the prime whose square divides~$N$.

\begin{lemma}  \label{5 and up lemma}
Let $N$ be a positive integer and $k\ge2$ an even integer.
If there exists a prime $p\ge5$ such that $p^2\mid N$, then $\Delta(k,N)>0$.
\end{lemma}

\begin{proof}
Since $s_0^*(N) \le 1$ for all positive integers $N$, we may simplify the inequality~\eqref{1312 ineq} to
\[
\Delta(k,N) \ge \frac12\nu_\infty^*(N) - \frac{13}{12}.
\]
But the fact that $p^2 \mid N$ implies that $\nu_\infty^*(p^2) \le \nu_\infty^*(N)$, and therefore
\[
\Delta(k,N) \ge \frac12 \nu_\infty^*(p^2) - \frac{13}{12} = \frac{p-1}2 - \frac{13}{12},
\]
which is positive thanks to the assumption $p\ge5$.
\end{proof}

\begin{lemma}  \label{3 lemma}
Let $N$ be a positive integer and $k\ge2$ an even integer.
If $9\mid N$, then $\Delta(k,N)>0$ unless $k=2$ and $N=9$.
\end{lemma}

\begin{proof}
The fact that $9\mid N$ implies that $s_0^*(9) \ge s_0^*(N)$ and $\nu_\infty^*(9) \le \nu_\infty^*(N)$, and hence Lemma~\ref{Deltakn ineq lemma} implies
\begin{align*}
\Delta(k,N) &\ge \frac{k-1}{12} N \big( 1 - s_0^*(9) \big) + \frac12\nu_\infty^*(9) - \frac{13}{12} = \frac{(k-1)N}{108} - \frac1{12}.
\end{align*}
The right-hand side is automatically positive when $(k-1)N>9$; given the assumption $9\mid N$, the only case left to check (since $k$ is a positive even integer) is $\Delta(2,9)=0$.
\end{proof}

\begin{lemma}  \label{2 lemma}
Let $N$ be a positive integer and $k\ge2$ an even integer.
If $4\mid N$, then $\Delta(k,N)>0$ unless $k=2$ and $N=4$.
\end{lemma}

\begin{proof}
The fact that $4\mid N$ implies that $s_0^*(4) \ge s_0^*(N)$ and $\nu_\infty^*(4) \le \nu_\infty^*(N)$, and hence Lemma~\ref{Deltakn ineq lemma} implies
\begin{align*}
\Delta(k,N) &\ge \frac{k-1}{12} N \big( 1 - s_0^*(4) \big) + \frac12\nu_\infty^*(4) - \frac{13}{12} = \frac{(k-1)N}{48} - \frac7{12}.
\end{align*}
The right-hand side is automatically positive when $(k-1)N>28$; given the assumption $4\mid N$, the only cases left to check (since $k$ is a positive even integer) are $\Delta(2,4)=-\frac12$ and $\Delta(2,8)=\Delta(2,12)=\Delta(2,16)=\Delta(2,20)=\Delta(2,24)=\Delta(2,28)=\Delta(4,4)=\Delta(4,8)=\Delta(6,4)=\Delta(8,4)=\frac12$.
\end{proof}

The proof of Theorem~\ref{main thm} is now complete.

\section{Bounds for the size of square divisors}  \label{bounds section}

Theorem~\ref{main thm} tells us that a single value $A(k,N)$ is enough to determine whether the number $N$ is squarefree or not. In this section, we show that even more detailed information can be obtained from $A(k,N)$: we can place upper and lower bounds upon the possible square factors of~$N$. We provide explicit upper and lower bounds for such divisors in Proposition~\ref{trig bounds prop}, and asymptotic versions of those bounds in Proposition~\ref{asymptotic d bounds prop}. The latter statement, in particular, makes it clear that these bounds are best when $A(k,N)$ is close to $G(k,N)$; in the course of the proof we will see that their difference cannot be significantly smaller than $\sqrt[3]N$ when $N$ is large (equation~\eqref{T AMGM} gives a precise inequality of this type). We end this section with an illustration of these bounds for the simplest example of a non-squarefree number~$N$.

\begin{definition}  \label{T0 T def}
For the rest of this section, given a positive integer $N$ and a positive even integer $k$, we will use the notation
\begin{align*}
T_0 = 12 \bigg( \Delta(k,N) + \frac{1}{2}-c_2(k)\legendre {-4} N-c_3(k)\legendre{-3} N \bigg).
\end{align*}
We see that $T_0$ is essentially a scaled version of $\Delta(k,N)$: it is easy to check that $|T_0-12\Delta(k,N)| \le 13$, and $T_0$ can be instantly computed from $\Delta(k,N)$ without requiring the factorization of~$N$. We also define
\[
T = \begin{cases}
T_0 +3, &\text{if } 3\mid k, \\
T_0 +7, &\text{if } 3\nmid k.
\end{cases}
\]
In particular, $T$ can be computed from a given value of $A(k,N)$ in polynomial time (since $G(k,N)$ is trivial to calculate).
\end{definition}

\begin{lemma}  \label{T ineq lemma}
For any positive integer $N$ and any positive even integer $k$,
\[
T \ge (k-1)N \big( 1 - s_0^{*}(N) \big) + 6\nu_{\infty}^{*}(N).
\]
\end{lemma}

\begin{proof}
Comparing Definition~\ref{T0 T def} and equation~\eqref{Deltakn formula}, we see that
\begin{align*}
T_0 &= (k-1)N \big( 1 - s_0^{*}(N) \big) + 6\nu_{\infty}^{*}(N) - 12c_2(k)\nu_2^{*}(N) - 12c_3(k)\nu_3^{*}(N) \\
&\ge (k-1)N \big( 1 - s_0^{*}(N) \big) + 6\nu_{\infty}^{*}(N) - \begin{cases}
3, &\text{if } 3\mid k, \\
7, &\text{if } 3\nmid k
\end{cases}
\end{align*}
by Definitions~\ref{c2 c3 def} and~\ref{nu2 nu3 def}; this inequality is equivalent to the statement of the lemma.
\end{proof}

\begin{definition}  \label{L theta def}
Given a positive integer $N$ and a positive even integer $k$, we will also use the notation
\begin{align*}
\L&= e^\gamma \log\log\sqrt N + \frac{2.50637}{\log\log\sqrt N} \\
\theta &= \arccos \bigg( 1-\frac{486(k-1)N}{\L ^2T^3} \bigg),
\end{align*}
where $\gamma\approx 0.577$ is Euler's constant; note that $\L$ is positive when $N\ge10$. (We will see that in our application, $\theta$ is always well defined.)
\end{definition}

\noindent We emphasize that $T$, $\L$, and $\theta$ all depend on $k$ and $N$, though we have suppressed that dependence from the notation for the sake of readability.

The following proposition gives explicit, easily computable upper and lower bounds (given a value $A(k,N)$) for integers whose square divides~$N$.

\begin{prop}  \label{trig bounds prop}
Let $N$ be a positive integer. If $d\ge27$ is an integer such that $d^2\mid N$, then
\[
\frac{\L T}{9}\cos\bigg(\frac{\theta}{3}-\frac{2\pi}{3}\bigg)+\frac{\L T}{18} < d < \frac{\L T}{9}\cos\frac{\theta}{3}+\frac{\L T}{18},
\]
where $\L$, $T$, and $\theta$ are as in Definitions~\ref{T0 T def} and~\ref{L theta def}.
\end{prop}

\begin{proof}
Since $d^2\mid N$, by Definition~\ref{mult fns def} we have
\[
\frac{d^2-1}{d^2}\ge \frac1{d^2} \prod_{p^e\parallel d} (p^{2e}-1) = \prod_{p^e\parallel d} \frac{p^{2e}-1}{p^{2e}} \ge \prod_{p\mid d} \frac{p^2-1}{p^2} = s_0^{*}(d^2)\ge s_0^{*}(N),
\]
and consequently $1-s_0^{*}(N) \ge 1/d^2$. Furthermore, if we let $D$ be the largest integer such that $D^2\mid N$, then again by Definition~\ref{mult fns def},
\[
\nu_{\infty}^{*}(N) = \phi(D) \ge \phi(d) > \frac d{e^\gamma\log\log d+2.50637/\log\log d} \ge \frac d\L,
\]
where $\L$ is as in Definition~\ref{L theta def};
here the middle inequality is an explicit upper bound for $d/\phi(d)$ by Rosser--Schoenfeld~\cite[Theorem 15]{RS}, and the last inequality is due to the fact that $d\le\sqrt N$ and that the function $e^\gamma\log\log x+2.50637/\log\log x$ is increasing for $x\ge27$.
Therefore by Lemma~\ref{T ineq lemma},
\begin{equation}  \label{T inequality}
T \ge (k-1)N \big( 1 - s_0^{*}(N) \big) + 6\nu_{\infty}^{*}(N) > \frac{(k-1)N}{d^2} + \frac{6d}\L
\end{equation}
or equivalently
\begin{equation}  \label{cubic inequality}
-\frac6\L d^3 + Td^2 - (k-1)N > 0.
\end{equation}

Consider the cubic polynomial $f(x) = -\frac6\L x^3 + Tx^2 - (k-1)N$. The value $f(0)$ is negative, while $f(x)$ is positive when $x$ is sufficiently negative; therefore $f(x)$ has a negative root. On the other hand, $f(d)$ is positive by equation~\eqref{cubic inequality}, while $f(x)$ is negative when $x$ is sufficiently positive. Therefore $f(x)$ has three real roots $x_0,x_1,x_2$. The trigonometric form of Cardano's formula (see~\cite[equation A1.23]{CCLH}) yields an exact expression for these three roots:
\[
x_j = \frac{\L T}{9}\cos\bigg( \frac{\theta}{3}-\frac{2j\pi}{3} \bigg)+\frac{\L T}{18}, \quad j=0,1,2,
\]
where $\theta$ is as in Definition~\ref{L theta def}. One can check that $x_2 \le 0 \le x_1 \le x_0$; since $d$ is positive, the inequality~\eqref{cubic inequality} forces $x_1 < d < x_0$, which is the statement of the lemma.
\end{proof}

The results of the previous proposition can be converted into asymptotic bounds whose sizes are easier to gauge (although less suited for explicit computation).

\begin{prop} \label{asymptotic d bounds prop}
Let $N$ be a positive integer. If $d\ge27$ is an integer such that $d^2\mid N$, then
\begin{multline*}
\sqrt{\frac{(k-1)N}{12\Delta(k,N)}} + O\bigg( \frac{kN}{\Delta(k,N)^2\log\log N} \bigg) \le d \\
< 2e^\gamma \Delta(k,N)\log\log N + O\bigg( \frac{\Delta(k,N)}{\log\log N} + \log\log N \bigg).
\end{multline*}
\end{prop}

\begin{proof}
We first claim that
\begin{equation}  \label{semi asymptotic bounds}
\sqrt{\frac{(k-1)N}{T}} + O\bigg( \frac{kN}{\L T^2} \bigg) \le d < \frac{\L T}{6}.
\end{equation}
The upper bound follows directly from $T>6d/\L$, which is a consequence of equation~\eqref{T inequality}, or from the right-hand inequality in Proposition~\ref{trig bounds prop}. As for the lower bound, we use the Puiseux series approximation
\begin{align*}
\cos \bigg( \frac{\arccos(1-x)}3 - \frac{2\pi}3 \bigg) &= -\frac12 + \sqrt{\frac x6} + O(x),
\end{align*}
with $x={486(k-1)N}/{\L^2T^3}$ (so that $\theta=\arccos(1-x)$ by Definition~\ref{L theta def}), in the left-hand inequality of Proposition~\ref{trig bounds prop}. We obtain
\begin{align*}
d &\ge \frac{\L T}{9} \bigg( {-}\frac12 + \sqrt{\frac{486(k-1)N/\L^2T^3}6} + O \bigg( \frac{486(k-1)N}{\L^2T^3} \bigg) \bigg) +\frac{\L T}{18} \\
&= \sqrt{\frac{(k-1)N}T} + O \bigg( \frac{kN}{\L T^2} \bigg)
\end{align*}
as claimed.

Note that equation~\eqref{T inequality} implies
\begin{equation}  \label{T AMGM}
T > \frac{(k-1)N}{d^2} + \frac{6d}\L \ge \sqrt[3]{\frac{243(k-1)N}{\L^2}}
\end{equation}
by calculus or a weighted arithmetic mean/geometric mean inequality. The hypotheses of the proposition force $N\ge27^2$, and the right-hand side of equation~\eqref{T AMGM} is an increasing function of $N$ in this range; from these inequalities (and $k\ge2$) we deduce that $T > 21$. In particular, since $|T-12\Delta(k,N)| \le 20$, we are justified in writing $T = 12\Delta(k,N) (1 + O(1/{\Delta(k,N)}))$. From Definition~\eqref{L theta def}, we may also write $\L = (1+O(1/{(\log\log N)^2})) e^\gamma \log\log N$. These last two approximations convert equation~\eqref{semi asymptotic bounds} into the asymptotic form asserted by the proposition.
\end{proof}

We illustrate this last proposition with an example. Suppose that $N=Ep^2$, where $p>3$ is prime and $E\equiv1\mod{12}$ is a squarefree number not divisible by~$p$. (For numbers encountered in practice that are not squarefree but have no square factors that are easily found through direct computation, this factorization type is by the far the most likely. The simplifying assumption $E\equiv1\mod{12}$ is solely for the purposes of exposition.) The various multiplicative functions in the definitions of $G(k,N)$ and $A(k,N)$ take the following values: $s_0^{*}(N) = 1 - \frac{1}{p^2}$ and $\nu_{\infty}^{*}(N) = p-1$, while $\legendre {-4} N = \legendre {-3} N = 1$ (since $N$ is also congruent to $1$ modulo~$12$) and $\nu_2^*(N) = \nu_3^*(N) = 0$. Consequently, taking $k=2$,
\[
G(2,N) = \frac{1}{12}Ep^2 - \frac12 - \frac14 - \frac13 \quad\text{and}\quad A(2,N) = \frac1{12}E(p^2-1)-\frac{1}{2}(p-1)
\]
and therefore
\[
\Delta(2,N) = \frac{E + 6p - 19}{12}.
\]
From this evaluation, we see that if $p\asymp N^\alpha$, then $\Delta(2,N) \asymp N^{1-2\alpha}$ when $\alpha\le\frac13$, while $\Delta(2,N) \asymp N^\alpha$ when $\alpha\ge\frac13$.

When $\alpha\le \frac13$, so that $\Delta(2,N) \asymp N^{1-2\alpha}$, the lower bound on $d=p$ in Proposition~\ref{asymptotic d bounds prop} is $\asymp N^\alpha$ while the upper bound is $\asymp N^{1-2\alpha}\log\log N$; in particular, $p$ is quite close to the lower bound. On the other hand, when $\alpha>\frac13$, so that $\Delta(2,N) \asymp N^{\alpha}$, the lower bound in Proposition~\ref{asymptotic d bounds prop} is $\asymp N^{(1-\alpha)/2}$ while the upper bound is $\asymp N^\alpha\log\log N$; in particular, $p$ is quite close to the upper bound. In either case, one of the two bounds is always rather sharp in this example. (The bounds, while remaining valid, can become less sharp if the squarefull part of $N$ is more complicated.)

\section{Factorization of the squarefull part}  \label{two values section}

Until now, we have investigated the consequences of having one calculated value of $A(k,N)$. Theorem~\ref{2 inputs thm} goes further, asserting that we can completely factor the squarefull part of a number $N$ with access to two calculated values of $A(k,N)$. After three preliminary lemmas, we prove Theorem~\ref{2 inputs thm} at the end of this section.

\begin{lemma}  \label{factor squarefull lemma}
Let $N>1$ be an integer. Given the values $s_0^*(N)$ and $\nu_\infty^*(N)$ (as in Definition~\ref{mult fns def}), the complete factorization of the squarefull part of $N$ can be found in probabilistic polynomial time.
\end{lemma}

\begin{proof}
Write $N=EL$ as the product of its squarefree part $E$ and its squarefull part $L$ with $(E,L)=1$, and note that we know the quantities $s_0^{*}(L)=s_0^{*}(N)$ and $\nu_{\infty}^{*}(L)=\nu_{\infty}^{*}(N)$.
We claim that it suffices to find a divisor $d>1$ of $L$ that we can factor completely. For if we have such a divisor $d$, then from its prime factors we can easily compute a factorization $N=bn$, where $(b,n)=1$ and the primes dividing $b$ are exactly the primes dividing~$d$. (Sometimes one writes $b=\gcd(d^\infty,N)$ to describe this factor.) We can then compute $s_0^*(n)=s_0^*(N)/s_0^*(b)$ and $\nu_\infty^*(n)=\nu_\infty^*(N)/\nu_\infty^*(b)$ from the known values $s_0^*(N)$ and $\nu_\infty^*(N)$ and directly from the definitions of $s_0^*(b)$ and $\nu_\infty^*(b)$, and then repeat recursively (setting $N=n$) until $n=1$. There are $o(\log N)$ prime factors of $N$ initially, which means that the number of divisions/multiplications needed in each calculation in this procedure, as well as the number of recursive calls to the procedure itself, are $\ll \log N$; and the integers that appear, along with the numerators and denominators of the rational numbers that appear, are all bounded by~$N$. (This utilization of the divisor $d$ is completely deterministic.)

To find such a divisor $d$, we simply set $d$ equal to the denominator of $s_0^*(L)$. This denominator cannot equal $1$ since $0<s_0^*(L)<1$ (here we use the fact that $L$ is squarefull, so that even during the recursion we always have $s_0^*(L)<1$), and by Definition~\ref{mult fns def} it is clearly a divisor of $\prod_{p\mid L} p^2$ which itself divides~$L$. On the other hand, note that $d\nu_{\infty}^{*}(L) = d \prod_{p^e\parallel L} (p-1)p^{\floor{e/2-1}}$ is a multiple of $d \prod_{p\mid d} (p-1)$ which in turn is a multiple of~$\phi(d)$. All that remains is to use the fact, well known to cryptographers (see~\cite[Section 10.4]{S}), that given a number $d$ and a multiple of $\phi(d)$, there is a probabilistic polynomial-time algorithm for factoring~$d$.
\end{proof}

Our general strategy, therefore, is to use two known values of $A(k,N)$ to determine the values $s_0^*(N)$ and $\nu_\infty^*(N)$, so that the above lemma can be applied. However, the definition of $A(k,N)$ also includes the two other multiplicative functions $\nu_2^{*}(N)$ and $\nu_3^{*}(N)$. In the next two lemmas we show that we can determine the values of these simpler functions directy from $A(k,N)$.

\begin{lemma}  \label{factor squarefull lemma 1.5}
Let $k$ be a positive even integer, and let $N$ be a positive integer.
\begin{enumerate}
\item Suppose that $9\mid N$ but $27\nmid N$. Then $N/9$ is squarefree if and only if
\begin{equation}  \label{N/9}
A(k,N) = \frac{2(k-1)}{27}N - 1 - c_3(K) \legendre{-3}{N/9}.
\end{equation}
\item Suppose that $4\mid N$ but $8\nmid N$. Then $N/4$ is squarefree if and only if
\begin{equation}  \label{N/4}
A(k,N) = \frac{k-1}{16}N - \frac12 - c_2(K) \legendre{-4}{N/4}.
\end{equation}
\end{enumerate}
\end{lemma}

\begin{proof}
By direct calculation, we may assume that $N\ge38$. Proposition~\ref{AkN formula} immediately implies both the equality~\eqref{N/9} when $\frac N9$ is squarefree and the equality~\eqref{N/4} when $\frac N4$ is squarefree, so it remains only to prove the converses. In part (a), the equality~\eqref{N/9} can be written as 
\[
\frac{k-1}{12}Ns_0^{*}(N)-\frac{1}{2}\nu_{\infty}^{*}(N)+c_3(k)\nu_3^{*}(N) = \frac{2(k-1)}{27}N - 1 - c_3(k) \legendre{-3}{N/9}
\]
(we know that $\nu_2^{*}(N)=0$ since $9\mid N$), or equivalently
\[
\frac{k-1}{12}N \bigg( \frac89 - s_0^{*}(N) \bigg) + \frac{\nu_{\infty}^{*}(N)+2}{2} = c_3(k) \bigg( \nu_3^{*}(N) + \legendre{-3}{N/9} \bigg) + 2.
\]
Suppose, for the sake of contradiction, that $\frac N9$ is not squarefree. Choose a prime $p$ such that $p^2\mid \frac N9$, and note that $p\ne3$ since $27\nmid N$. Then $s_0^*(N) \le \frac89(1-\frac1{p^2})$ and $\nu_{\infty}^{*}(N) \ge 2(p-1)$ (and $k-1\ge1$), and so
\[
\frac1{12}N \frac8{9p^2} + p \le c_3(k) \bigg( \nu_3^{*}(N) + \legendre{-3}{N/9} \bigg) + 2 \le \frac83.
\]
However, the left-hand side is at least $(\frac N2)^{1/3}$ (for any positive real number $p$, by an easy calculus exercise). Therefore we must have $N \le 2(\frac83)^3 < 38$, a contradiction. The proof of part (b) is similar, starting from the given equality
\[
\frac{k-1}{12}Ns_0^{*}(N)-\frac{1}{2}\nu_{\infty}^{*}(N)+c_2(k)\nu_2^{*}(N) = \frac{k-1}{16}N - \frac12 - c_2(k) \legendre{-4}{N/4}
\]
and eventually deducing that
\[
\frac1{12}N \frac3{4p^2} + \frac p2 \le c_2(k) \bigg( \nu_2^{*}(N) + \legendre{-4}{N/4} \bigg) + 1 \le \frac32,
\]
forcing $N\le32$ which is again a contradiction.
\end{proof}

\begin{lemma}  \label{factor squarefull lemma 1.75}
Let $k$ be a positive even integer, and let $N$ be a positive integer. Given the value $A(k,N)$, we can determine the values $\nu_2^*(N)$ and $\nu_3^*(N)$.
\end{lemma}

\begin{proof}
When $4\nmid N$ and $9\nmid N$, Theorem~\ref{main thm} and the known value $A(k,N)$ allow us to decide whether $N$ is squarefree, which is all that is needed to calculate $\nu_2^*(N)$ and $\nu_3^*(N)$ in this case. When
$9\mid N$, we immediately know that $\nu_2^*(N)=0$, and if $27\mid N$ then $\nu_3^*(N)=0$ as well; if $27\nmid N$, Lemma~\ref{factor squarefull lemma 1.5}(a) allows us to determine whether $\frac N9$ is squarefree, which is what is needed to calculate $\nu_3^*(N)$. Finally, when $4\mid N$, we immediately know that $\nu_3^*(N)=0$, and if $8\mid N$ then $\nu_2^*(N)=0$ as well; if $8\nmid N$, Lemma~\ref{factor squarefull lemma 1.5}(b) allows us to determine whether $\frac N4$ is squarefree, which is what is needed to calculate $\nu_2^*(N)$.
\end{proof}

\begin{proof}[Proof of Theorem~\ref{2 inputs thm}]
Define $A^*(k,N) = A(k,N) - c_2(k)\nu_2^{*}(N)-c_3(k)\nu_3^{*}(N)$, and note that $A^*(k,N)$ can be calculated easily from $A(k,N)$ by Lemma~\ref{factor squarefull lemma 1.75}.
In this notation, Proposition~\ref{AkN formula} implies that
\begin{align*}
A^*(k_1,N) &= \frac{k_1-1}{12}Ns_0^{*}(N)-\frac{1}{2}\nu_{\infty}^{*}(N) \\
A^*(k_2,N) &= \frac{k_2-1}{12}Ns_0^{*}(N)-\frac{1}{2}\nu_{\infty}^{*}(N).
\end{align*}
This system of two linear equations in the two unknown quantities $s_0^{*}(N)$ and $\nu_{\infty}^{*}(N)$ can be easily solved in (deterministic) polynomial time, giving quantities that are trivial to calculate from the hypothesized known values:
\begin{align*}
s_0^{*}(N) &= \frac {12\big(A^*(k_2,N)-A^*(k_1,N)\big)}{(k_2-k_1)N} \\
\nu_{\infty}^{*}(N) &= \frac{2\big(A^*(k_2,N)(k_1-1) - A^*(k_1,N)(k_2-1)\big)}{k_2-k_1}.
\end{align*}
Therefore, by Lemma~\ref{factor squarefull lemma}, we can obtain the factorization of the squarefull part of $N$ in probabilistic polynomial time, as claimed.
\end{proof}

\section{Testing for primality} \label{prm thm section}

In this section we establish Theorem~\ref{prm test thm}, concerning the function $B(k,N)$ given in Definition~\ref{B def}. Similarly to $A(k,N)$, an exact formula for $B(k,N)$ as a linear combination of multiplicative functions of $N$ was given by the second author~\cite{M}. The following proposition records this formula with just enough precision for our current purposes.

\begin{prop} \label{B prop}
For any positive integer $N$ and any positive even integer $k$,
\begin{equation}  \label{B formula}
B(k,N)=\frac{k-1}{12}Ns_0^{\#}(N)-\frac{1}{2}\nu_{\infty}^{\#}(N)+c_2(k)\nu_2^{\#}(N)+c_3(k)\nu_3^{\#}(N) + \delta_2(k) \mu(N),
\end{equation}
where $\mu$ is the M\"obius mu-function; $c_2$, $c_3$, and $\delta_2$ are as in Definition~\ref{c2 c3 def}; and $s_0^{\#}$, $\nu_{\infty}^{\#}$, $\nu_2^{\#}$, and $\nu_3^{\#}$ are certain multiplicative functions satisfying:
\begin{enumerate}
\item $Ns_0^\#(N) = \phi(N)$ when $N$ is squarefree;
\item $\nu_{\infty}^{\#}(p)=0$ for every prime $p$;
\item the only possible values of $\nu_2^{\#}(N)$ are $0$ and $\pm 2^\ell$ for some integer $0\le \ell\le\omega(N)$, where $\omega(N)$ is the number of distinct prime factors of~$N$, and similarly for $\nu_3^{\#}(N)$;
\item $\nu_2^\#(p) = \legendre{-4}p - 1$ and $\nu_3^\#(p) = \legendre{-3}p - 1$ for every prime $p$, where these Kronecker symbols are as in Definition~\ref{-4N -3N def}.
\end{enumerate}
\end{prop}

\begin{proof}
The given formula appears as~\cite[Theorem 1]{M} (in which $B(k,N)$ is denoted by $g_0^{\#}(k,N)$). The exact definitions of the multiplicative functions $s_0^{\#}$, $\nu_2^{\#}$, $\nu_3^{\#}$, and $\nu_{\infty}^{\#}(N)$ are given in~\cite[Definition 1$'$]{M}, from which the listed properties follow immediately.
\end{proof}

\begin{cor}  \label{gstar sfree prop}
Let $k$ be a positive even integer. When $N$ is squarefree,
\begin{equation} \label{BkN sfree}
B(k,N) = \frac{(k-1)\phi(N)}{12} - \frac{\delta_1(N)}2 + c_2(k) \nu_2^\#(N) +  c_3(k) \nu_3^\#(N) + \delta_2(k) \mu(N),
\end{equation}
where $\delta_1$ is as in Definition~\ref{c2 c3 def}.
In particular,
\begin{align}
B(k,1) &= \frac{k-7}{12} + c_2(k) + c_3(k) + \delta_2(k) \label{Bk1} \\
B(k,p) &= \frac{(k-1)(p-1)}{12} + c_2(k) \nu_2^\#(p) + c_3(k) \nu_3^\#(p) - \delta_2(k) \label{Bkp} \\
B(k,pq) &= \frac{(k-1)(p-1)(q-1)}{12} + c_2(k) \nu_2^\#(p)\nu_2^\#(q) + c_3(k) \nu_3^\#(p)\nu_3^\#(q) + \delta_2(k) \label{Bkpq}
\end{align}
for any distinct primes $p$ and~$q$.
\end{cor}

\begin{proof}
These identities are direct consequences of Proposition~\ref{B prop}, using the assumption that $N$ is squarefree and the fact that $\nu_{\infty}^{\#}$, $\nu_2^{\#}$, and $\nu_3^{\#}$ are multiplicative functions.
\end{proof}

\begin{cor} \label{vanishers}
Let $k$ be a positive even integer. For any prime $p$, we have $B(k,p)>0$, except for the pairs
\[
(k,p) = (2,2),\, (2,3),\, (2,5),\, (2,7),\, (2,13),\, (4,2),\, (4,3),\, (6,2),\, \text{or } (12,2)
\]
for which $B(k,p)=0$. Similarly, for any distinct primes $p$ and $q$, we have $B(k,pq)>0$, except that $B(2,6) = B(2,10) = B(2,22) = 0$.
\end{cor}

\begin{proof}
Since $|c_2(k) \nu_2^\#(p) + c_3(k) \nu_3^\#(p) - \delta_2(k)| \le \frac14\cdot2+\frac13\cdot2+1=\frac{13}6$ by Proposition~\ref{B prop}(d), the inequality $B(k,p)>0$ follows from equation~\eqref{Bkp} when $\frac{(k-1)(p-1)}{12} > \frac{13}6$, or equivalently when $(k-1)(p-1) > 26$; the finitely many values remaining can be computed individually. Similarly, from equation~\eqref{Bkpq} we see that $\big| B(k,pq) - \frac{(k-1)(p-1)(q-1)}{12} \big| \le \big( \frac14\cdot4+\frac13\cdot4+1 \big) = \frac{10}3$; therefore $B(k,pq)$ is positive when $(k-1)(p-1)(q-1) > 40$, and the remaining values can be checked individually.
\end{proof}

We remark that these values can also be found on the LMFDB, an online database of information about specific $L$-functions, modular forms, and related objects: \url{http://www.lmfdb.org/ModularForm/GL2/Q/holomorphic}

The last fact we need about the function $B(k,N)$ is its relationship to $A(k,N)$.
The isomorphism classes of automorphic representations associated with the space of weight-$k$ cusp forms on~$\Gamma_0(N)$ are in one-to-one correspondence with weight-$k$ Hecke newforms on~$\Gamma_0(d)$ as $d$ ranges over the positive divisors of~$N$. In particular, comparing the cardinalities of these sets results in the well-known convolution formula (see~\cite[first displayed equation on page 311]{M})
\begin{equation} \label{ABrel}
A(k,N) = \sum_{d\mid N} B(k,d).
\end{equation}
In particular, if $N$ is prime then by Definition~\ref{H def} and Theorem~\ref{main thm} (since all primes are squarefree),
\begin{align*}
H(k,N) = G(k,N) - B(k,1) &= A(k,N) - B(k,1) \\
&= \big( B(k,N) + B(k,1) \big) - B(k,1) = B(k,N),
\end{align*}
which establishes the first assertion of Theorem~\ref{prm test thm}. The remaining assertions of Theorem~\ref{prm test thm} are established in the following four lemmas.

\begin{lemma} \label{nsf lemma}
Let $k\ge2$ be an even integer. If $N$ is a positive integer that is not squarefree, then $H(k,N)>B(k,N)$, except that $H(2,4)<B(2,4)$ and $H(2,9)=B(2,9)$.
\end{lemma}

\begin{proof}
By Definition~\ref{H def} and Theorem~\ref{main thm}, as long as $(k,N) \ne (2,4)$ and $(k,N) \ne (2,9)$ we have
\[
H(k,N) = G(k,N) - B(k,1) > A(k,N) - B(k,1) = \sum_{\substack{d\mid N \\ d>1}} B(k,d) \ge B(k,N)
\]
by equation~\eqref{ABrel},
where the last inequality is valid since $N>1$ and the dimensions $B(k,d)$ are nonnegative. On the other hand, the values $H(2,4)=-\frac{1}{2}$ and $B(2,4)=H(2,9)=B(2,9)=0$ can be calculated directly.
\end{proof}

\begin{lemma} \label{proper divisor lemma}
Let $N\ge2$ be a squarefree positive integer and $k\ge2$ an even integer. If there exists a nontrivial divisor $d_0$ of $N$ (that is, $1<d_0<N$ and $d_0\mid N$) such that $B(k,d_0)>0$, then $H(k,N)>B(k,N)$.
\end{lemma}

\begin{proof}
By Definition~\ref{H def} and Theorem~\ref{main thm},
\begin{align*}
H(k,N) = G(k,N) - B(k,1) &= A(k,N) - B(k,1) \\
&= \sum_{\substack{d\mid N \\ d>1}} B(k,d) \ge B(k,N) + B(k,d_0) > B(k,N)
\end{align*}
by equation~\eqref{ABrel}, again since the dimensions $B(k,d)$ are nonnegative.
\end{proof}

\begin{lemma} \label{sf lemma}
Let $N\ge2$ be a squarefree positive integer and $k\ge4$ an even integer. If $N$ is not prime, then $H(k,N)>B(k,N)$, except that $H(4,6)=B(4,6)$.
\end{lemma}

\begin{proof} 
Since $N$ is squarefree and not prime, there exist distinct primes $p<q$ dividing~$N$. By Corollary~\ref{vanishers} we know that $B(k,q)$ is positive, in which case $H(k,N)>B(k,N)$ by Lemma~\ref{proper divisor lemma} with $d_0=q$. The only exception to this argument is in the case $k=4$ and $N=6$, where both $B(4,2)$ and $B(4,3)$ vanish; we verify directly that $H(4,6)=1=B(4,6)$.
\end{proof}

\begin{lemma} \label{sf 2 lemma}
Let $N\ge2$ be a squarefree positive integer. If $N$ is not prime, then $H(2,N)>B(2,N)$ unless $N=6$, $10$, $14$, $15$, $21$, $26$, $35$, $39$, $65$, or~$91$, in which case $H(2,N)=B(2,N)$.
\end{lemma}

\begin{proof} 
First suppose that $N$ has at least three distinct prime factors. Then at least two of these primes $p,q$ must be odd; let $d_0=pq$ be their product, which in particular is not equal to $6$, $10$, or~$22$. By Corollary~\ref{vanishers}, we see that $B(k,d_0)>0$, and therefore $H(2,N)>B(2,N)$ by Lemma~\ref{proper divisor lemma}.

The only remaining case is when $N=pq$ is the product of two distinct primes. If one of these primes is not in the set $\{2,3,5,7,13\}$, then let $d_0$ be that prime; by Corollary~\ref{vanishers}, we see that $B(k,d_0)>0$, and therefore $H(2,N)>B(2,N)$ by Lemma~\ref{proper divisor lemma}. Otherwise, both $p$ and $q$ are in the set $\{2,3,5,7,13\}$, which leads to the $\binom52=10$ values of $N$ listed in the statement of the lemma; and direct computation verifies that $H(2,N)=B(2,N)$ in all ten cases.
\end{proof}

The proof of Theorem~\ref{prm test thm} is now complete.

\section{Full factorization} \label{fact N section}

The last result that remains to be established is Theorem~\ref{3 inputs thm}, on factorization the integer $N$ completely using two values of $A(k,N)$ and one value of $B(k,N)$. Fortunately, after the work in the earlier sections, the proof of this theorem is quite brief.

\begin{lemma} \label{fact N lemma}
Let $k$ be a positive even integer. Let $E$ be a positive squarefree integer and $L$ a positive squarefull number such that $\gcd(E,L)=1$, and set $N=EL$. Suppose that we know the value $B(k,N)$, the complete factorization of $L$, and the values $\nu_2^{\#}(N)$, $\nu_3^{\#}(N)$, and $\mu(N)$. Then we can find the complete factorization of $E$ (and hence of $N$) in probabilistic polynomial time.
\end{lemma}

\begin{proof} 
Since the case $E=1$ is trivial, we may assume that $E>1$. In this case, there exists at least one prime that divides $N$ to the first power, which gives $\nu_{\infty}^{\#}(N)=0$ by Proposition~\ref{B prop}(b). Then, by equation~\eqref{B formula},
\[
Ns_0^{\#}(N) = \frac{12}{k-1} \bigg( B(k,N) - c_2(k)\nu_2^{\#}(N) - c_3(k)\nu_3^{\#}(N) - \delta_2(k) \mu(N) \bigg),
\]
where every term on the right-hand side is known by assumption. We can then compute $Es_0^{\#}(E)=Ns_0^{\#}(N)/Ls_0^{\#}(L)$ from the known value $Ns_0^{\#}(N)$ and directly from the definition of $Ls_0^{\#}(L)$. On the other hand, since $E$ is squarefree, Proposition~\ref{B prop}(a) tells us that $Es_0^{\#}(E)=\phi(E)$. Therefore we know the values $E=N/L$ and $\phi(E)$, and hence we can factor $E$ in probabilistic polynomial time (as mentioned in the proof of Lemma~\ref{factor squarefull lemma}).
\end{proof}

\begin{proof}[Proof of Theorem~\ref{3 inputs thm}]
From the two values $A(k_1,N)$ and $A(k_2,N)$, we can obtain (by Theorem~\ref{2 inputs thm}) the factorization of $N=EL$ into its squarefree and squarefull parts, along with the complete factorization of $L$, in probabilistic polynomial time. Since we also know the value $B(k,N)$, the only obstacle to applying Lemma~\ref{fact N lemma} is our ignorance of the values $\nu_2^{\#}(N)$, $\nu_3^{\#}(N)$, and $\mu(N)$.

However, by Proposition~\ref{B prop}(c), the only possible values for $\nu_2^{\#}(N)$ and $\nu_3^{\#}(N)$ are $0$ or $\pm 2^\ell$ for some integer $2^\ell \le N$, and of course there are only three possible values for $\mu(N)$. Thus there are $\ll(\log N)^2$ possible values for the triple $\big( \nu_2^{\#}(N), \nu_3^{\#}(N),\mu(N) \big)$. We simply use each of these possible values (in parallel or in turn) and attempt to apply Lemma~\ref{fact N lemma}. The computations using incorrect values might result in errors or infinite loops or even incorrect factorizations; but the computation using the correct values will (with high probability) yield the correct factorization---and any proposed factorization can be quickly checked for correctness. Therefore these many computations will indeed yield the factorization of $N$ in probabilistic polynomial time.
\end{proof}

Though we have made no attempt to optimize the running time of the algorithms presented in this paper, it is natural to speculate how one could avoid the need for $O(\log^2N)$ parallel computations in this last proof. Morally speaking, from four values $B(k_j,N)$ for distinct even numbers $k_1,\dots,k_4$, we can use linear algebra to try to solve the system of four equations arising from the formula~\eqref{B formula} for the four unknown values $Ns_0^{\#}(N)$, $\nu_2^{\#}(N)$, $\nu_3^{\#}(N)$, and $\mu(N)$. Indeed, we should even be able to manage with three values $B(k_j,N)$, since the term $\delta_2(k) \mu(N)$ disappears when $L>1$ (or when $k>2$), and there are only three possible values for $\delta_2(k) \mu(N)$ in any event.

While this approach does work for many triples of weights $k_1,k_2,k_3$, it does not work in every possible circumstance, because the resulting system of three linear equations in $Ns_0^{\#}(N)$, $\nu_2^{\#}(N)$, and $\nu_3^{\#}(N)$ might be degenerate. For example, if $(k_1,k_2,k_3) = (12m, 12m+4, 12m+8)$, then any one value among $B(k_1,N),B(k_2,N),B(k_3,N)$ can be calculated from the other two (the middle value is the average of the first and last values, for instance). On the other hand, for certain pairs of weights---for example, if $k_1\equiv k_2\mod{12}$ and $k_1,k_2>2$---it is possible to solve for the crucial value $Ns_0^{\#}(N)$, even though the individual values $\nu_2^{\#}(N)$ and $\nu_3^{\#}(N)$ are still entangled with each other. In practice, from the explicit formula~\eqref{B formula}, it is an easy matter to determine how much information can be extracted from the values $B(k_j,N)$ for any given weights~$k_j$.

\red{\section*{Acknowledgments}}

\red{The authors thank the anonymous referee for several helpful comments, and particularly for pointing out the connection to primality tests that led to Theorem~\ref{prm test thm}. The second author's work is partially supported by a National Sciences and Engineering Research Council of Canada Discovery Grant.}

\end{document}